\documentclass[]{interact}
\usepackage{latexsym}
\usepackage{anyfontsize} 
\usepackage{color}
\usepackage{hyperref}
\usepackage{url}
\usepackage{breakurl}
\newcommand{\bburl}[1]{\textcolor{blue}{\url{#1}}}

\makeatletter
\newcommand{\monthyear}[1]{%
  \def\@monthyear{\uppercase{#1}}}
\newcommand{\volnumber}[1]{%
  \def\@volnumber{\uppercase{#1}}}
\makeatother

\theoremstyle{plain}
\numberwithin{equation}{section} 
\newtheorem{thm}{Theorem}[section] 
\newtheorem{theorem}[thm]{Theorem}
\newtheorem{lemma}[thm]{Lemma}

\numberwithin{table}{section} 
\numberwithin{figure}{section}

\begin{document}

\monthyear{Month Year}
\volnumber{Volume, Number}
\setcounter{page}{1}

\title{Self-Convolutions of Generalized Narayana Numbers}

\author{
\name{Greg Dresden\textsuperscript{a}\thanks{Email address: dresdeng@wlu.edu}, Yuechen Xiao\textsuperscript{b}, and Guanzhang Zhou\textsuperscript{a}}
\affil{\textsuperscript{a}Department of Mathematics, Washington and Lee University,
Lexington, VA, 24450, USA;
\textsuperscript{b}Shanghai Starriver Bilingual School,
       2588 Jindu Rd, Minhang District, Shanghai, China}
}

\maketitle

\begin{abstract}
	For the Fibonacci numbers $F_n$, we have the self-convolution formula 
    $5 \sum_{i=0}^n F_i F_{n-i} = (2n)F_{n+1} - (n+1)F_n$. We find the corresponding self-convolution formula for the Narayana numbers $R_n$ which satisfy $R_n = R_{n-1} + R_{n-3}$, and then generalize it to the $k$-step Narayana numbers $\mathcal{R}_n$
    with order-$k$ recurrence formula 
    $\mathcal{R}_n = \mathcal{R}_{n-1} + \mathcal{R}_{n-k}$.
\end{abstract}


\section{Introduction}

We begin with the Fibonacci numbers, defined as
\[ F_0 =0, \ F_1 = 1, \ \mbox{and} \ 
F_{n} = F_{n-1} + F_{n-2} \  \mbox{for $n \geq 2$.}
\]
For the following collection of  formulas, the sum on the left is called the {\em self-convolution} of the Fibonacci numbers.
\begin{align}\label{e.fibcon}
5 \sum_{i=0}^n F_i F_{n-i} &= (2n)F_{n+1} - (n+1)F_n \\[-1.9ex]
&= (n-1)F_n + (2n)F_{n-1} \nonumber \\[0.3ex]
&= (n-1)F_{n+1} + (n+1)F_{n-1} \nonumber \\[0.3ex]
&= n L_n - F_n. \nonumber
\end{align}
Most of these can be found  at \href{https://oeis.org/A001629}{A001629} on the On-Line Encyclopedia of Integer Sequences (OEIS) \cite{oeis} and can also be found in various forms at 
\cite[Identity 58]{BQ},
\cite[Theorem 1]{Robbins}, 
\cite[p.~183]{Vajdo}, \cite[Corollary 1]{Zhang}, and in many other papers. 
 We note that the $L_n$ in the last formula 
    represents the $n$th Lucas number.

For the self-convolution of the Lucas numbers $L_n$, we have 
\begin{equation}\label{e.luccon}
\sum_{i=0}^n L_i L_{n-i} = (n+1)L_n + 2F_{n+1}
\end{equation}
which  appears in  \cite[Identity 57]{BQ} and 
\cite[Corollary 10]{Szakacs2}. The Fibonacci numbers and the Lucas numbers are two specific (and famous) examples of second-order recurrences. A good collection of authors have studied convolutions of linear recurrences. 
Adegoke, Akerele, and Frontczak in \cite{Adegoke} use variations on Binet's formula to give self-convolution formulas for  general second-order recurrence sequences; their paper includes the identities  (\ref{e.fibcon}) and (\ref{e.luccon}), above, along with many others. Nacin \cite{Nacin} found tilings proofs for convolutions of Fibonacci numbers with Pell numbers and with Jacobsthal numbers. For the convolution of Fibonacci polynomials, see \cite{Abd}.

 Aside from \cite[Theorem 6]{DresdenWang} which generalizes  (\ref{e.luccon}) to order-$k$ Lucas-type sequences, and 
 \cite[equation (52)]{Rabinowitz} which covers third-order recurrences, and 
 \cite{Gry} which covers the tribonacci numbers, there are not many results on the  self-convolutions of higher-order recurrence sequences. This paper aims to partially fill that gap. In particular, we will consider the order-$k$ recurrence sequences that are natural generalizations of the Narayana sequence (which is itself a generalization of the Fibonacci numbers).

 \section{The Narayana Sequence}

We now define the Narayana numbers $R_n$ as 
\begin{equation}\label{e.def.Rn}
R_0 = 0, \ R_i = 1 \mbox{\ for $i=1,2$, and }  R_{n} = R_{n-1} +  R_{n-3} \mbox{\ for $n \geq 3$.}
\end{equation}
We can verify by hand that the following self-convolution equation seems to be true: 
\begin{equation}\label{e.Narayana31}
31\sum_{i=0}^n R_i R_{n-i}=
9 (n + 1) R_{n + 2} - 3 (n + 3) R_{ n + 1} 
- 2 (n + 2) R_n.
\end{equation}
Equation (\ref{e.Narayana31}) is indeed true, and  can be derived from equation (52) in Rabinowitz's 1996 article \cite{Rabinowitz}, with a little bit of effort. Rabinowitz's equation covers a general third-order recurrence defined as 
\[
X_0 = X_1 = 0, X_2 = 1, \quad \mbox{and} \quad X_n = pX_{n-1} + qX_{n-2} + rX_{n-3} \quad \mbox{for $n \geq 3$}.
\]
If we take $p=1$, $q=0$, and $r=1$ to match  (\ref{e.def.Rn}),  we eventually arrive at 
\begin{equation}\label{e.Rab31}
\sum_{i=0}^n X_i X_{n-i} = \Big( 6(n-2)X_{n+1} - 2nX_{n-1} + 3(n+1)X_{n-2}\Big)/31.
\end{equation}
Since $X_n = R_{n-1}$ thanks to the slightly shifted initial values, then  with a bit more manipulation we can transform  (\ref{e.Rab31}) into  (\ref{e.Narayana31}).

\section{The 4-step Narayana numbers}

We can think of the Narayana numbers as being a ``3-step" sequence because the last term in the recurrence  (\ref{e.def.Rn}) is $R_{n-3}$. With this in mind, we 
 also define the ``4-step" Narayana numbers $S_n$ 
as 
\begin{equation*}
S_0 = 0, \ S_i = 1 \mbox{\ for $1\leq i \leq 3$, and }  
S_{n} = S_{n-1} +  S_{n-4} \mbox{\ for $n \geq 4$.}
\end{equation*}
We can verify (with some difficulty) that 
\begin{multline}\label{e.Narayana41}
\ \ \ \ \ \ \ \ \ \ \ \ \ \ \ \ \ \ 
283\sum_{i=0}^n S_i S_{n-i}=
64 (n + 2) S_{n + 3} - 16 (n + 5) S_{ n + 2} \\
- 12(n+4)S_{n+1} - 9(n+3)S_n.\ \ \ \ \ \ \ \ \ \ \ \ \ \ \ 
\end{multline}
The constants 5, 31, and 283
that appear on the 
left of each of the self-convolution formulas (\ref{e.fibcon}), (\ref{e.Narayana31}), and (\ref{e.Narayana41}) might seem to be random, but we note that $\sqrt{5}$ is part of the golden ratio \cite{BQ} related to the Fibonacci numbers, and likewise $\sqrt{31}$ for the ``supergolden" ratio \cite{Crilly} associated with the Narayana numbers. Another insight comes from noting that 5, -31, and -283 are the discriminants of $1-x-x^k$ for $k=2$, 3, and 4. For our purposes, we make the observation that all three
 can be expressed as the sum of two powers:
\begin{equation*}
    5 = 2^2 + 1^1,  \qquad 
    31 = 3^3 + 2^2, \qquad \mbox{and} \qquad
    283 = 4^4 + 3^3.
\end{equation*}
Likewise, there seems to be a pattern to the coefficients on the right of each of those self-convolution formula. We can make this more clear if we re-write 
 (\ref{e.Narayana41}) as follows:
\begin{multline}\label{e.Narayana41a}
(4^4 + 3^3)\sum_{i=0}^n S_i S_{n-i}=
        4^3 (n + 2) S_{n + 3} \\ 
- \Big(3^2(n+3)S_n + 3\cdot 4(n+4)S_{n+1} + 4^2  (n + 5) S_{ n + 2}\Big).\ \
\end{multline}
This inspires us to re-write  (\ref{e.Narayana31}) in the same style as (\ref{e.Narayana41a}):
\begin{equation}\label{e.Narayana51a}
    (3^3 + 2^2)\sum_{i=0}^n R_i R_{n-i}=
        3^2 (n + 1) R_{n + 2}  
- \Big( 2^1(n+2)R_{n} + 3^1 (n + 3) R_{ n + 1}\Big).
\end{equation}
We do the same with our Fibonacci formula in  (\ref{e.fibcon}), giving us 
\begin{equation}\label{e.Narayana21a}
    (2^2 + 1^1)\sum_{i=0}^n F_i F_{n-i}=
        2^1 (n+0) F_{n + 1}  
- \Big(  1^0 (n + 1) F_{ n}\Big).
\end{equation}
We can now see that there is a common theme for all three self-convolution formulas. As we show in the next section, we can generalize this beyond the 3-step and 4-step Narayana numbers $R_n$ and $S_n$, respectively.

\section{Main Result}

We define the $k$-step Narayana numbers $\mathcal{R}_n$ as 
\begin{equation}\label{e.RnInitial}
\mathcal{R}_0 = 0, \ \mathcal{R}_i = 1 \mbox{\ for $1\leq i \leq k-1$, and }  
\mathcal{R}_{n} = \mathcal{R}_{n-1} +  \mathcal{R}_{n-k} \mbox{\ for $n \geq k$.}
\end{equation}
When $k=2$ we recapture the Fibonacci numbers $F_n$, and for $k=3$ and $k=4$ we obtain $R_n$ and $S_n$, respectively. To be precise, we should probably use the notation $\mathcal{R}_{n}^{(k)}$
instead of $\mathcal{R}_{n}$ to indicate that these are 
$k$-step Narayana numbers, but to keep our formulas a bit cleaner we will simply use $\mathcal{R}_{n}$ with the understanding that these numbers are dependent on $k$.

We have made a deliberate choice with our initial conditions for $\mathcal{R}_n$: namely, 
$\mathcal{R}_0 = 0$ and then $\mathcal{R}_i = 1$ for $i$ from $1$ to $k-1$. Not only will this choice give us a 
particularly simple generating function as seen in (\ref{e.Rngf}), but also this choice will produce the particularly nice patterns we see in the right-hand sides of 
(\ref{e.Narayana41a}) and (\ref{e.Narayana51a}). Other choices for initial conditions (and yes, we have tried out several of them) do not give such pleasant formulas.

Our main result is as follows.

\begin{theorem}\label{t.t1}
    For $k\geq 2$ fixed, and with  $\mathcal{R}_n$ representing the $k$-step Narayana numbers
    defined above in (\ref{e.RnInitial}), we have
\begin{multline}\label{e.theorem3}
\left(k^k + (k-1)^{k-1} \right)\sum_{i=0}^n  \mathcal{R}_i \mathcal{R}_{n-i} \\ = \ 
k^{k-1} (n + k-2) \mathcal{R}_{n + k-1}
- \sum_{j=0}^{k-2}\left(k^j\cdot (k-1)^{k-2-j}\right) (n+k+j-1) \mathcal{R}_{n+j}. 
\end{multline}
\end{theorem}

\noindent It is important to note that this theorem applies for any and all $k\geq 2$. When $k=2$ then the $2$-step Narayana numbers are just the Fibonacci numbers $F_n$ as explained earlier, and indeed  (\ref{e.theorem3}) with $k=2$ and with $F_n$ in place of $\mathcal{R}_n$ will give us 
(\ref{e.Narayana21a}). Likewise, when $k=3$ then we have the
``traditional" Narayana numbers $R_n$, and likewise 
(\ref{e.theorem3}) with $k=3$ and with $R_n$ in place of $\mathcal{R}_n$ will give us 
(\ref{e.Narayana51a}).
Furthermore, if we want a formula for the self-convolution of the ``6-step" Narayana numbers which we define as 
\begin{equation*}
U_0 = 0, \ U_i = 1 \mbox{\ for $1\leq i \leq 5$, and }  
U_{n} = U_{n-1} +  U_{n-6} \mbox{\ for $n \geq 6$,}
\end{equation*}
then Theorem \ref{t.t1} with $k=6$ tells us that 
\[
(6^6 + 5^5) \sum_{i=0}^n U_i U_{n-i} = 6^5(n+4)U_{n+5} - \sum_{j=0}^4 6^j 5^{4-j} (n+5+j)U_{n+j}.
\]

\section{Technical Lemma}

Here is a  lemma that we will need for the proof of our Theorem \ref{t.t1}.

\begin{lemma}\label{lemma2}
    For $k \geq 2$ and $m \geq 0$ both  integers,   we have
\begin{equation}\label{e.lemma2}
    k^{k-2-m} (k-1)^m \sum_{i=0}^m \left(\frac{k}{k-1}
\right)^i (k+i) = k^{k-1} (m+1).   
\end{equation}
\end{lemma}

\begin{proof}
    For convenience, we set $\theta=k/(k-1)$ and so the left-hand side of (\ref{e.lemma2})  is 
\begin{equation}\label{e.lemma21}
k^{k-2-m} (k-1)^m \sum_{i=0}^m \theta^i (k+i).
\end{equation}
This is actually a telescoping sum, and to see this
we re-write $\theta^i (k+i)$ as follows:
\begin{align*}
\theta^i (k+i) &= \theta^i(k + ik - ik + i)\\ 
            &= \theta^i((i+1)k-i(k-1)).    
\end{align*}
We now factor out $(k-1)$ from the right, and we use that $k/(k-1) = \theta$, to give us 
\begin{equation*}
    \theta^i (k+i) = (k-1)\theta^i((i+1)\theta - i).
\end{equation*}
Finally, we multiply through by $\theta^i$ on the right to give us 
\begin{equation*}
 \theta^i (k+i) = (k-1)((i+1)\theta^{i+1} - i\theta^i).
\end{equation*}
This allows us to re-write (\ref{e.lemma21}) as 
\begin{equation*}
k^{k-2-m} (k-1)^m \sum_{i=0}^m (k-1)((i+1)\theta^{i+1} - i\theta^i),
\end{equation*}
and so this telescoping sum collapses to give us 
\begin{equation*}
k^{k-2-m} (k-1)^m \cdot (k-1)((m+1)\theta^{m+1}.
\end{equation*}
Since $\theta=k/(k-1)$, then this expression collapses once again to give us 
\[
k^{k-1}(m+1),
\]
as desired.
\end{proof}

\section{Proof of Theorem \ref{t.t1}}

We now have all the pieces we need to prove our main result.

\begin{proof}[Proof of Theorem \ref{t.t1}]
    For convenience, we will label the three parts of  (\ref{e.theorem3}) as follows:
\begin{align}
A_n &= \left(k^k + (k-1)^{k-1} \right)\sum_{i=0}^n  \mathcal{R}_i \mathcal{R}_{n-i}, \label{e.An}\\
B_n &=  
k^{k-1} (n + k-2) \mathcal{R}_{n + k-1}, \label{e.Bn}\\
C_n &=  \sum_{j=0}^{k-2}\left(k^j\cdot (k-1)^{k-2-j}\right) (n+k+j-1) \mathcal{R}_{n+j}. \label{e.Cn}
\end{align}
To show that $A_n = B_n - C_n$, we will show that their generating functions, which we will write as $A(x), B(x)$, and $C(x)$, satisfy
\begin{equation*}
A(x) = B(x) - C(x). 
\end{equation*}

\noindent{\bf Generating function for the first part}. We begin with $A(x)$. Since the generating function for 
$\mathcal{R}_n$ is 
\begin{equation}\label{e.Rngf}
\sum_{n=0}^{\infty} \mathcal{R}_n x^n = \frac{x}{1-x-x^k},
\end{equation}
then  the Cauchy product rule \cite[p.~36]{Wilf} tells us that the generating function for the self-convolution 
$\sum_{i=0}^n \mathcal{R}_i \mathcal{R}_{n-i}$
will be given by 
\[
\sum_{n=0}^{\infty}\left(\sum_{i=0}^n \mathcal{R}_i \mathcal{R}_{n-i}\right) x^n = \left( \frac{x}{1-x-x^k}\right)^2 = 
\frac{x^2}{(1-x-x^k)^2}.
\]
So, our generating function $A(x)$ for the sequence 
$A_n$ from  (\ref{e.An}) is 
\begin{equation}\label{e.A(x)boxed}  
 A(x) = \left(k^k + (k-1)^{k-1}\right) \frac{x^2}{(1-x-x^k)^2}.
\end{equation}

\noindent{\bf Generating function for the second part}. Next, we look at $B(x)$. From the definition of $B_n$ in  (\ref{e.Bn}) we have that its generating function $B(x)$ is 
\begin{equation*}
    B(x) = k^{k-1}
    \sum_{n=0}^{\infty} (n+k-2) \mathcal{R}_{n+k-1}x^n. 
\end{equation*}
We notice that this is a term-by-term derivative of another, simpler power series, as shown here:
\begin{equation*}
    B(x) = k^{k-1}
    \sum_{n=0}^{\infty} \mathcal{R}_{n+k-1}\cdot \frac{1}{x^{k-3}} \frac{d}{dx} x^{n+k-2}. 
\end{equation*}
We re-arrange the terms in this sum to give us 
\begin{equation*}
    B(x) = \frac{k^{k-1}}{x^{k-3}}\cdot  \frac{d}{dx}
    \sum_{n=0}^{\infty} \mathcal{R}_{n+k-1} x^{n+k-2}. 
\end{equation*}
This is almost, but not quite, what we want, because the subscript for $\mathcal{R}_{n+k-1}$ is not quite a perfect match for the exponent in $x^{n+k-2}$. An easy fix is to multiply and divide by $x$, giving us 
\begin{equation*}
    B(x) = \frac{k^{k-1}}{x^{k-3}}\cdot  \frac{d}{dx}
    \left( \frac{1}{x}\sum_{n=0}^{\infty} \mathcal{R}_{n+k-1} x^{n+k-1}\right). 
\end{equation*}
We now notice that the sum in the above equation starts at $n=0$ with $\mathcal{R}_{k-1} x^{k-1}$ and so if we re-index the sum to start with $\mathcal{R}_{0} x^{0}$ and then subtract the unwanted terms from $\mathcal{R}_{0} x^{0}$ up to $\mathcal{R}_{k-2} x^{k-2}$,
we have 
\begin{equation*}
    B(x) = \frac{k^{k-1}}{x^{k-3}}\cdot  \frac{d}{dx}
    \left( \frac{1}{x}\sum_{n=0}^{\infty} \mathcal{R}_{n}x^n  - \frac{1}{x}\sum_{i=0}^{k-2} \mathcal{R}_{i}x^i \right). 
\end{equation*}
The first sum is simply the generating function for $\mathcal{R}_n$ as seen in  (\ref{e.Rngf}). For the second sum, we note from  the initial conditions in (\ref{e.RnInitial}) that  
$\mathcal{R}_0 = 0$ and $\mathcal{R}_i = 1$ for $i$ between $1$ and $k-2$, and so we now have 
\begin{equation*}
    B(x) = \frac{k^{k-1}}{x^{k-3}}\cdot  \frac{d}{dx}
    \left( \frac{1}{1-x-x^k}  - \frac{1}{x}\sum_{i=1}^{k-2} x^i \right). 
\end{equation*}
We now distribute the $1/x$ into the sum on the right of the above equation, giving us (after re-indexing)
\begin{equation*}
    B(x) = \frac{k^{k-1}}{x^{k-3}}\cdot  \frac{d}{dx}
    \left( \frac{1}{1-x-x^k}  - \sum_{i=0}^{k-3} x^i \right). 
\end{equation*}
When we apply the derivative to the right-hand side, we get 
\begin{equation*}
    B(x) = \frac{k^{k-1}}{x^{k-3}}
    \left( \frac{1+kx^{k-1}}{(1-x-x^k)^2}  - \sum_{i=0}^{k-3} ix^{i-1} \right). 
\end{equation*}
and after distributing, we get 
\begin{equation}\label{e.B(x)boxed}
    B(x) = 
    \frac{k^{k-1}+k^k x^{k-1}}{x^{k-3}(1-x-x^k)^2}  - \frac{k^{k-1}}{x^{k-3}} \sum_{i=0}^{k-3} ix^{i-1}. 
\end{equation}

\noindent{\bf Generating function for the third part}. Finally, we turn our attention to $C(x)$. From the definition of $C_n$ in  (\ref{e.Cn}) we have that its generating function $C(x)$ is 
\begin{equation*}
    C(x) = \sum_{n=0}^{\infty} \sum_{j=0}^{k-2}\left(k^j\cdot (k-1)^{k-2-j}\right) (n+k+j-1) \mathcal{R}_{n+j} x^n.
\end{equation*}
When we switch the order of summation, we have 
\begin{equation*}
    C(x) = \sum_{j=0}^{k-2}\left(k^j\cdot (k-1)^{k-2-j}\right) \sum_{n=0}^{\infty} (n+k+j-1) \mathcal{R}_{n+j} x^n,
\end{equation*}
and just as with $B(x)$ earlier we notice that 
the power series on the right is a term-by-term derivative of another, simpler power series, as given here:
\begin{equation*}
    C(x) = \sum_{j=0}^{k-2}\left(k^j\cdot (k-1)^{k-2-j}\right) \sum_{n=0}^{\infty} \frac{1}{x^{k+j-2}} \frac{d}{dx}  \mathcal{R}_{n+j} x^{n+k+j-1}.
\end{equation*}
We re-arrange the terms to give us 
\begin{equation*}
    C(x) = \sum_{j=0}^{k-2}\frac{k^j\cdot (k-1)^{k-2-j}}{x^{k+j-2}} \cdot\frac{d}{dx}  \sum_{n=0}^{\infty}  \mathcal{R}_{n+j} x^{n+k+j-1}.
\end{equation*}
Once again we notice that this is almost, but not quite, what we want, because the subscript for $\mathcal{R}_{n+j}$ is not quite a perfect match for the exponent in $x^{n+k+j-1}$. An easy fix is to pull out $x^{k-1}$, giving us 
\begin{equation*}
    C(x) = \sum_{j=0}^{k-2}\frac{k^j\cdot (k-1)^{k-2-j}}{x^{k+j-2}} \cdot\frac{d}{dx}  x^{k-1}\sum_{n=0}^{\infty}  \mathcal{R}_{n+j} x^{n+j}.
\end{equation*}
We note that each sum on the right of the above equation starts at $n=0$ with $\mathcal{R}_{j} x^{j}$, and so if we re-index each sum to start with $\mathcal{R}_{0} x^{0}$ and then subtract the unwanted terms from $\mathcal{R}_{0} x^{0}$ up to $\mathcal{R}_{j-1} x^{j-1}$, we have 
\begin{equation}\label{e.63}
    C(x) = \sum_{j=0}^{k-2}\frac{k^j\cdot (k-1)^{k-2-j}}{x^{k+j-2}} \cdot\frac{d}{dx} \left( x^{k-1}\sum_{n=0}^{\infty}  \mathcal{R}_{n} x^{n} - x^{k-1}\sum_{i=0}^{j-1}  \mathcal{R}_{i} x^{i}\right),
\end{equation}
with the understanding that when $j=0$ the last sum on the right is an empty (zero) sum. 

Now, the first sum inside the parentheses on the right 
of  (\ref{e.63}) 
is simply the generating function for $\mathcal{R}_n$ as seen in  (\ref{e.Rngf}), and for the second sum we recall once again from  (\ref{e.RnInitial}) that 
$\mathcal{R}_0 = 0$ and $\mathcal{R}_i=1$ for $i$ between $1$ and $k-1$. This means that our last sum on the right of  (\ref{e.63}) actually starts at $i=1$ (because, again, 
$\mathcal{R}_0 = 0$ by definition) and has just powers of $x$ without coefficients. 
Putting this all together, we see that   (\ref{e.63}) becomes
\begin{equation}\label{e.64}
    C(x) = \sum_{j=0}^{k-2}\frac{k^j\cdot (k-1)^{k-2-j}}{x^{k+j-2}} \cdot\frac{d}{dx} \left( \frac{x^{k}}{1-x-x^k} - x^{k-1}\sum_{i=1}^{j-1}  x^{i}\right).
\end{equation}
We note that the last sum on the right is an empty sum (and hence is zero) when $j=0$ or $j=1$. We now put the $x^{k-1}$ back into the last sum on the right, giving us 
\begin{equation*}
    C(x) = \sum_{j=0}^{k-2}\frac{k^j\cdot (k-1)^{k-2-j}}{x^{k+j-2}} \cdot\frac{d}{dx} \left( \frac{x^{k}}{1-x-x^k} - \sum_{i=1}^{j-1}  x^{k+i-1}\right).
\end{equation*}
When we take the derivative and simplify, we have 
\begin{equation*}
    C(x) = \sum_{j=0}^{k-2}\frac{k^j\cdot (k-1)^{k-2-j}}{x^{k+j-2}} \left( \frac{kx^{k-1}-(k-1)x^k}{(1-x-x^k)^2 } 
    - \sum_{i=1}^{j-1}  (k+i-1) x^{k+i-2}\right).
\end{equation*}
After canceling the common $x^{k-1}$ term everywhere, we have 
\begin{equation*}
    C(x) = \sum_{j=0}^{k-2}\frac{k^j\cdot (k-1)^{k-2-j}}{x^{j-1}} \left( \frac{k-(k-1)x}{(1-x-x^k)^2 } 
    - \sum_{i=1}^{j-1}  (k+i-1) x^{i-1}\right).
\end{equation*}
We re-index that last sum by replacing $i$ with $i+1$, giving us  
\begin{equation*}
    C(x) = \sum_{j=0}^{k-2}\frac{k^j\cdot (k-1)^{k-2-j}}{x^{j-1}} \left( \frac{k-(k-1)x}{(1-x-x^k)^2 } 
    - \sum_{i=0}^{j-2}  (k+i) x^{i}\right).
\end{equation*}
After distributing and re-organizing, we have 
\begin{equation}\label{e.C(x)almostboxed}
    C(x) = \frac{kx-(k-1)x^2}{(1-x-x^k)^2 } \sum_{j=0}^{k-2} \left(\frac{k}{x}\right)^j (k-1)^{k-2-j}
    - x\sum_{j=0}^{k-2} \left(\frac{k}{x}\right)^j(k-1)^{k-2-j} \sum_{i=0}^{j-2} (k+i) x^{i}.
\end{equation}
For the first sum in the above equation, we use the identity
\begin{equation*}
\sum_{j=0}^{m} a^j b^{m-j} = \frac{a^{m+1} - b^{m+1}}{a-b}
\end{equation*}
to write
\begin{equation*}
    \sum_{j=0}^{k-2} \left(\frac{k}{x}\right)^j (k-1)^{k-2-j} = \frac{(k/x)^{k-1} - (k-1)^{k-1}}{k/x - (k-1)}
    = \frac{x^2\left((k/x)^{k-1} - (k-1)^{k-1}\right)}{kx - (k-1)x^2},
\end{equation*}
and when we substitute this into  (\ref{e.C(x)almostboxed}) and simplify, we get
\begin{equation}\label{e.C(x)almostboxed2}
    C(x) = \frac{x^2\left((k/x)^{k-1} 
    - (k-1)^{k-1}\right)}{(1-x-x^k)^2 } 
    - x\sum_{j=0}^{k-2} \left(\frac{k}{x}\right)^j(k-1)^{k-2-j} \sum_{i=0}^{j-2} (k+i) x^{i}.
\end{equation}
We now multiply the top and bottom of the first term by $x^{k-3}$ to give us 
\begin{equation}\label{e.C(x)almostboxed3}
      C(x) = \frac{ k^{k-1} 
    - (k-1)^{k-1}x^{k-1}}{x^{k-3}(1-x-x^k)^2 } 
    - x\sum_{j=0}^{k-2} \left(\frac{k}{x}\right)^j(k-1)^{k-2-j} \sum_{i=0}^{j-2} (k+i) x^{i}.
\end{equation}

We now turn our attention to the double sum on the right of  (\ref{e.C(x)almostboxed3}). We move the $x^j$ term inside the second sum, we  bring that $x^i$ term into the denominator, we note that the first sum can start at $j=2$ instead of at $j=0$, and then we bring out the second summation, giving us 
\begin{equation}\label{e.C(x)almostboxed4}
    C(x) = \frac{k^{k-1} - (k-1)^{k-1} x^{k-1}}{x^{k-3}(1-x-x^k)^2 } 
    - x\sum_{j=2}^{k-2}\sum_{i=0}^{j-2} k^j (k-1)^{k-2-j}  \frac{(k+i)}{x^{j-i}}.
\end{equation}
We wish to re-index this double sum, and so let us first describe the values for $i$ and $j$ that appear in the double sum:
\begin{center}
\begin{tabular}{rl}
 $j=2$: & $i=0.$  \\
 $j=3$: & $i=0,1.$  \\
 $j=4$: & $i=0,1,2.$  \\
 $j=5$: & $i=0,1,2,3.$  \\
 $\vdots\ \ \ \ $ & \\
 $j=k-2$: & $i=0,1,2,3, \dots, k-4.$  
 \end{tabular}
\end{center}
If we now let $w=j-i$, we see that $w$ runs from $w=2$ (which covers the pairs $(j=2,i=0)$ and $(j=3,i=1)$ and $(j=4,i=2)$, and so on) up to $w=k-2$ which covers only the pair $(j=k-2, i=0)$. So, we can cover all of this by letting $i$ run from $i=0$ to $i=k-2-w$, and so when we re-write our double sum in terms of $w$ and $i$, we have 
\begin{equation}\label{e.C(x)almostboxed5}
    C(x) = \frac{k^{k-1} - (k-1)^{k-1} x^{k-1}}{x^{k-3}(1-x-x^k)^2 } 
    - x\sum_{w=2}^{k-2}\sum_{i=0}^{k-2-w} 
    k^{w+i} (k-1)^{k-2-w-i}  \frac{(k+i)}{x^{w}}.
\end{equation}
When we pull out some terms from the inner sum, we have 
\begin{equation}\label{e.C(x)almostboxed6}
    C(x) = \frac{k^{k-1} - (k-1)^{k-1} x^{k-1}}{x^{k-3}(1-x-x^k)^2 } 
    - x\sum_{w=2}^{k-2} \frac{k^w(k-1)^{k-2-w}}{x^{w}} \sum_{i=0}^{k-2-w} \left(\frac{k}{k-1}\right)^i (k+i).
\end{equation}
We re-index once more, letting $m=k-2-w$ so that $m$ runs from $0$ to $k-4$, giving us 
\begin{equation}\label{e.C(x)almostboxed7}
    C(x) = \frac{k^{k-1} - (k-1)^{k-1} x^{k-1}}{x^{k-3}(1-x-x^k)^2 } 
    - x\sum_{m=0}^{k-4} \frac{k^{k-2-m}(k-1)^{m}}{x^{k-2-m}} \sum_{i=0}^{m} \left(\frac{k}{k-1}\right)^i (k+i).
\end{equation}
We now apply Lemma \ref{lemma2}, giving us 
\begin{equation}\label{e.C(x)almostboxed8}
    C(x) = \frac{k^{k-1} - (k-1)^{k-1} x^{k-1}}{x^{k-3}(1-x-x^k)^2 } 
    - x\sum_{m=0}^{k-4} \frac{1}{x^{k-2-m}} k^{k-1}(m+1).
\end{equation}
Cleaning up a bit, this gives us 
\begin{equation}\label{e.C(x)almostboxed9}
    C(x) = \frac{k^{k-1} - (k-1)^{k-1} x^{k-1}}{x^{k-3}(1-x-x^k)^2 } 
    - \frac{k^{k-1}}{x^{k-3}}\sum_{m=0}^{k-4} (m+1)x^m.
\end{equation}
That last sum is empty for $k<4$, but otherwise it starts with  $1 + 2x + 3x^2$ and ends with $(k-3)x^{k-4}$, so we can re-index it by $i=m+1$ and write it as 
\begin{equation}\label{e.C(x)boxed}
    C(x) = \frac{k^{k-1} - (k-1)^{k-1} x^{k-1}}{x^{k-3}(1-x-x^k)^2 } 
    - \frac{k^{k-1}}{x^{k-3}}\sum_{i=0}^{k-3} ix^{i-1}.
\end{equation}

\noindent{\bf Conclusion of proof}. When we compare our formula for $C(x)$ in (\ref{e.C(x)boxed}) with our formula for $B(x)$ in  (\ref{e.B(x)boxed}),
we note that 
\[
B(x) - C(x) = \frac{k^{k} x^{k-1}}{x^{k-3}(1-x-x^k)^2} -
 \frac{-(k-1)^{k-1}x^{k-1}}{x^{k-3}(1-x-x^k)^2} = 
\frac{\left( k^{k} + (k-1)^{k-1}\right) x^2}{(1-x-x^k)^2}, 
\]
which is a perfect match for $A(x)$ in  (\ref{e.A(x)boxed}), as desired. 
\end{proof}

\section*{Acknowledgments} The authors express their gratitude to Pioneer Academics and to Washington \& Lee University for providing the resources for this collaborative effort, and to our anonymous referee who gave us a much simpler proof of Lemma \ref{lemma2} along with many other helpful  comments that greatly improved this paper.

\section*{Disclosure Statement}
No  conflict of interest has been reported by the authors.

\medskip
\noindent MSC2020: 11B39

\end{document}